\documentclass[11pt,a4paper,reqno]{amsart}
\usepackage{amsthm,amsmath,amsfonts,amssymb,amsxtra,appendix,bookmark,dsfont,bm,mathrsfs,amstext,amsopn,mathrsfs,mathtools,comment,cite,color}
\usepackage[latin1]{inputenc}
\usepackage{dsfont}
\usepackage{hyperref}
\usepackage{comment,cite}
\usepackage{esvect} 
\usepackage{esint}
\usepackage{pst-all}
\usepackage{pstricks}
\usepackage{graphicx}
\usepackage{esvect} 
\usepackage{macros}
\newcommand{\no}[1]{ \left| \! \left| #1 \right| \! \right|_{L^2} }

\title[Unique continuation and the HK theorem]{Unique continuation for \\ many-body Schr\"odinger operators \\ and the Hohenberg-Kohn theorem{*}}

\author[L. Garrigue]{Louis Garrigue}
\address{CEREMADE, Universit\'e Paris-Dauphine, PSL Research University, F-75016 Paris, France} 
\email{garrigue@ceremade.dauphine.fr}\date{January 2019 \\ $\phantom {lla,}$*\,This is a corrected version of the published article \cite{Garrigue18}.}

\begin{document}

\begin{abstract} We prove the strong unique continuation property for many-body Schr\"odinger operators with an external potential and an interaction potential both in $L^p_{\rm{loc}}(\er{d})$, where $p > \max(2d/3,2)$, independently of the number of particles. With the same assumptions, we obtain the Hohenberg-Kohn theorem, which is one of the most fundamental results in Density Functional Theory. 
\end{abstract}

\maketitle

Density Functional Theory (DFT) is one of the most successful methods in quantum physics and chemistry to simulate matter at the microscopic scale \cite{DreGro90,CanDefKutLeBMad03,EngDre11}. It is a very active field of research, applied to very diverse physical situations, going from atoms and small molecules to condensed matter systems \cite{Jones15}.

One of the basis of DFT is due to Hohenberg and Kohn in 1964 \cite{HohKoh64}, who showed that in equilibrium, the knowledge of the ground state density alone is sufficient to characterize the system. In other words, all the information of a quantum system is contained in its ground state one-particle density. The Hohenberg-Kohn theorem was precised by Lieb in \cite{Lieb83b}, who emphasized that it relies on a unique continuation property (UCP) for the many-particle Hamiltonian.

A typical (strong) unique continuation result \cite{Tataru04} is that if a wavefunction $\p$ vanishes sufficiently fast at one point and solves Schr\"odinger's equation $H\p=0$, then $\p=0$.
Unique continuation properties began to be developped by Carleman in \cite{Carleman39} and, today, a broad range of results exists when the operator is $H = -\Delta + V(x)$, with $V$ in some $L^p_{\tx{loc}}$ space. A famous result of Jerison and Kenig \cite{JerKen85} covers the case $p=d/2$ in dimension $d$. It was later improved by Koch and Tataru in \cite{KocTat01}.

Unfortunately, these results are not well adapted to the situation of Schr\"odinger operators describing $N$ particles, which are defined on $\er{dN}$. In order to apply the existing results, one would need assumptions on the potentials depending on $N$. To the best of our knowledge, two works, due to Georgescu \cite{Georgescu80} and Schechter-Simon \cite{SchSim80}, provide a unique continuation property for many-particle Hamiltonians with an assumption on the potentials independent of $N$. However, they require the wavefunction to vanish on an open set (weak UCP), and for the Hohenberg-Kohn theorem strong UCP is needed.

Recently, Laestadius, Benedicks and Penz \cite{LaeBenPen17} have proved the first strong UCP result for many-body operators using ideas of Kurata \cite{Kurata97} and Regbaoui \cite{Regbaoui01}, but they need extra assumptions on the negative part of ${2V+ x \cdot \na V}$, which naturally appears in the Virial identity. In \cite{Zhou17}, Zhou used the result of Schechter and Simon to state a weak form of the Hohenberg-Kohn theorem, but this was already implicit in the work of Lieb \cite{Lieb83b}. We refer to \cite{Lieb83b,Kryachko05,CanDefKutLeBMad03,Levy10,Zhou12,Lammert18,PinBokLud07,EngEng83} for a discussion on the importance of the unique continuation principle for the Hohenberg-Kohn theorem.

In this article, we provide the first strong UCP for many-body operators in $L^p$ spaces and deduce the first complete proof of the Hohenberg-Kohn theorem in these spaces. Our proof mainly uses the method of Georgescu \cite{Georgescu80}, together with a Carleman estimate proved in \cite{Garrigue19}. We also use ideas from Figueiredo-Gossez \cite{FigGos92} to pass from the vanishing of $\p$ on a set of positive measure to the vanishing to infinite order at one point. In short, we can handle any number $N$ of particles living in $\er{d}$, with potentials in $L^p_{\rm{loc}}(\er{d})$ with $p > \max(2d/3,2)$.
We deduce the Hohenberg-Kohn theorem with similar assumptions.

\subsection*{Acknowledgement}
I warmly thank Mathieu Lewin, my PhD advisor, for having supervised me during this work. This project has received funding from the European Research Council (ERC) under the European Union's Horizon 2020 research and innovation programme (grant agreement MDFT No 725528).

\section{Main results}
\subsection{Strong unique continuation property}
We denote by $B_R$ the ball of radius $R$ centered at the origin. 
Our main result is the following.

\begin{theorem}[Strong UCP]\label{mainthm}
Let $V \in L^2_{\rm{loc}}(\er{n})$ such that for some $\delta > 0$ and for every ${R >0}$, there exists $c_R \geq 0$ such that for any $u \in H^2(\er{n})$,
\begin{equation}\label{main}
	\int_{B_R} \abs{V}^2 \abs{u}^2 \leq \ep_{\delta,n} \int_{\er{n}} \abs{(-\Delta)^{\frac{3}{4} - \delta} u}^2 + c_R \int_{\er{n}} \abs{u}^2,
\end{equation}
	where $\ep_{\delta,n}$ is a constant depending only on $\delta$ and on the dimension $n$. Let $\p \in H_{\rm{loc}}^2(\er{n})$ be a solution to $-\Delta \p + V \p = 0$. If $\p$ vanishes on a set of positive measure or if it vanishes to infinite order at a point, then ${\p=0}$.
\end{theorem}

The constant $\ep_{\delta,n}$ depends on the best constant of the Carleman inequality \eqref{carl} which we are going to use later.
We recall that $\p$ vanishes to infinite order at $x_0 \in \er{n}$ when for all $k \geq 1$, there is a $c_k$ such that
 \begin{equation*}
	 \int_{\abs{x-x_0} < \ep} \abs{\p}^2 \d x < c_k \ep^k,
 \end{equation*}
 for every $\ep <1$.

The assumption \eqref{main} can be rewritten in the sense of operators, in the form $$ \abs{V}^2 \indic_{B_R} \leq \ep_{\delta,n} (-\Delta)^{\f{3}{2} - 2\delta} +c_R. $$ This is satified if $V \in L^{p}_{\rm{loc}}(\er{n})$ with $p > \max(2n/3,2)$ (see the proof of Corollary \ref{th2}). However, the condition \eqref{main} has a better behavior with respect to the dimension than a condition in $L^p$ spaces. It is more appropriate to deal with $N$-body operators for which $n=dN$, as we will see. 
We denote by $B_R(x)$ the ball of radius $R$ and centered on $x \in \er{n}$. Assumption \eqref{main} is equivalent to saying that for any $x \in \er{n}$, there exists $c_x$ such that
	\beq \label{as2} \abs{V}^2 \indic_{B_1(x)} \leq \ep_{\delta,n}' (-\Delta)^{\f{3}{2}-2\delta} + c_x \qquad \tx{in $\er{n}$}. \eeq
	Indeed \eqref{as2} follows from \eqref{main} by taking $R = \abs{x} +1$ whereas the converse statement can be obtained by (fractional) localization, e.g. as in \cite[Lemma A.1]{LenLew10}.
We have stated our main result in $\R^n$ for simplicity, but there is a similar statement in a connected domain $\Omega$. One should then replace \eqref{main} by \eqref{as2} with small balls $B_R(x)\subset\Omega$. Our proof is really local in space.

Following Simon in \cite[section C.9]{Simon80}, we conjecture that Theorem \ref{mainthm} holds under the weaker condition $$\abs{V}\indic_{B_R} \leq  \ep_n (-\Delta) + c_R,$$ with $\p$ in $H^1_{\rm{loc}}(\er{n})$. 
A weak UCP was proved by Schechter and Simon in \cite{SchSim80} using estimates from Protter \cite{Protter60}, but with the stronger hypothesis $$\abs{V}^2 \indic_{B_R} \leq  \ep_n (-\Delta) +c_R.$$ Our Theorem \ref{mainthm} improves the weak UCP of Georgescu in \cite{Georgescu80}, which has an assumption similar to \eqref{main}. He used the estimate from Theorem 8.3.1 of \cite{Hormander63}, due to H\"ormander and we instead use a new Carleman estimate that we proved in \cite{Garrigue19}.

\subsection{Application to $N$-body operators}
We consider $N$ particles in $\er{d}$, submitted to an external potential $v$ and interacting with an even two-body potential $w$. The corresponding $N$-body Hamiltonian takes the form 
\begin{equation}\label{ham}
	H^N(v) = - \sum_{i=1}^N \Delta_{x_i} + \sum_{i=1}^{N} v(x_i) + \sum_{1 \leq i \sle j \leq N} w(x_i-x_j),
 \end{equation}
 on $L^2(\er{dN})$.
In order to ensure that the total potential
\begin{equation}\label{potintro}
V(x_1, \dots ,x_N) \df \sum_{i=1}^{N} v(x_i) + \sum_{1 \leq i \sle j \leq N} w(x_i-x_j),
\end{equation}
satisfies the assumption \eqref{main} in $\er{dN}$, it is sufficient that $v$ and $w$ satisfy \eqref{main} in $\er{d}$, but with an $\ep$ that can be taken as small as we want.

\begin{corollary}[UCP for many-body Schr\"odinger operators]\label{th2}
	Assume that the potentials satisfy 
\begin{equation} \label{borne}
	   \abs{v}^2 \indic_{B_R} + \abs{w}^2 \indic_{B_R} \leq \ep_{\delta,d,N} (-\Delta)^{\f{3}{2}-2\delta} + c_{R} \qquad \tx{ in $\er{d}$},
\end{equation}
	for some $\delta >0$ and for all $R>0$, where $\ep_{\delta,d,N}$ is a small constant depending only on $\delta$, $d$ and $N$. For instance ${v,w \in L^p_{\rm{loc}}(\er{d})}$ with $p> \max(2d/3)$.
	Let $\p \in H_{\rm{loc}}^2(\er{dN})$ be a solution to $H^N(v) \p = 0$. If $\p$ vanishes on a set of positive measure or if it vanishes to infinite order at a point, then ${\p=0}$.
\end{corollary}

\subsection{Hohenberg-Kohn theorem}
The one-particle density of a wavefunction $\p$ is defined as
\begin{equation*}
	\ro_{\p}(x) \df \sum_{i=1}^N \int \abs{\p(x_1,...,x_{i-1},x,x_{i+1},...,x_N)}^2 \d x_1 \cdots \d x_{i-1} \d x_{i+1} \cdots \d x_N.
 \end{equation*}
From Corollary \ref{th2}, we can deduce the following version of the Hohenberg-Kohn theorem.

\begin{theorem}[Hohenberg-Kohn]\label{hkthm}
	Let $w, v_1, v_2 \in (L^{p}+L^{\ii})(\er{d},\reals)$, with $p > \max(2d/3)$. If there are two normalised eigenfunctions $\p_1$ and $\p_2$ of $H^N(v_1)$ and $H^N(v_2)$, corresponding to the first eigenvalues, and such that $\ro_{\p_1} = \ro_{\p_2}$, then there exists a constant $c$ such that $v_1 = v_2 + c$. 
\end{theorem}

The exact same theorem is valid if we take spin into account and assume that $\p_1, \p_2$ are the first eigenfunctions of $H^N(v_1), H^N(v_2)$ in any subspace invariant by the two operators. In particular the theorem applies to bosons and fermions. Our result covers the physical case of Coulomb potentials as in \cite{Zhou12}. However, in this situation, eigenfunctions are real analytic on an open set of full measure, and the argument is much easier.

We recall the proof from \cite{HohKoh64,Lieb83b} for the convenience of the reader. 
\begin{proof}
	We denote by $\ro \df \ro_{\p_1} = \ro_{\p_2}$ the common density. Since $\p_1$ is the ground state for $v_1$, then 
	 \begin{equation*} 
		 E_1 \hspace{-0.02cm} \df\hspace{-0.03cm} \ps{\p_1,H^N(v_1) \p_1}\hspace{-0.03cm} \leq\hspace{-0.03cm} \ps{\p_2,H^N(v_1) \p_2}\hspace{-0.03cm} = \hspace{-0.03cm}\ps{\p_2, H^N(v_2)\p_2} +\hspace{-0.02cm} \int_{\er{d}} \ro (v_1 - v_2).
	 \end{equation*}

	We also have 
	 \begin{equation*}
		 E_2 \df \ps{\p_2,H^N(v_2) \p_2} \leq \ps{\p_1,H^N(v_1) \p_1} + \int_{\er{d}} \ro (v_2 - v_1). 
	 \end{equation*}
	 Hence $E_1-E_2= \int_{\er{d}} \ro (v_1 - v_2)$ and $\ps{\p_2,H^N(v_1) \p_2} = E_1$, so $\p_2$ is a ground state for $H^N(v_1)$, and $H^N(v_1) \p_2\hspace{-0.03cm} =\hspace{-0.03cm} E_1 \p_2$. 
	 Together with ${H^N(v_2) \p_2\hspace{-0.03cm} =\hspace{-0.03cm} E_2 \p_2}$, this gives 
	 \begin{equation*}
		 \pa{E_1-E_2 + \sum_{i=1}^N (v_2-v_1)(x_i)} \p_2 = 0. 
	 \end{equation*}
	 Since, by Corollary \ref{th2}, the normalised function $\p_2$ cannot vanish on a set of positive measure, we get 
	 \begin{equation*}
		  E_1-E_2 + \sum_{i=1}^N (v_2-v_1)(x_i) = 0
	 \end{equation*}
almost everywhere. Integrating this relation over $x_2, \dots, x_N$ in a bounded domain we conclude, as wanted, that $v_1-v_2 = c$. Eventually, using again $E_1-E_2= \int_{\er{d}} \ro (v_1 - v_2)$ yields $c = (E_1-E_2)/N$. 
\end{proof}

The rest of the paper is devoted to the proof of our main results.

\section{Proof of Theorem \ref{mainthm}}

\subsection*{Step 1. Vanishing on a set of positive measure implies vanishing to infinite order at one point.}

We will need to pass from $\p$ vanishing on a set of positive measure, which is the needed hypothesis for the Hohenberg-Kohn theorem, to $\p$ vanishing to infinite order at a point, which is the usual hypothesis for strong unique continuation. We reformulate here Proposition 3 of \cite{FigGos92} with slightly weaker assumptions.

\begin{proposition}[Figueiredo-Gossez \cite{FigGos92}] \label{prop}
	Let $V \in L^1_{\rm{loc}} (\er{n})$ such that for every $R>0$, there exist $a_R<1$ and $c_R >0$ such that 
\begin{equation}\label{fig}
	\abs{V} \indic_{B_R} \leq a_R (-\Delta) +c_R.
 \end{equation}
If $\p \in H_{\rm{loc}}^1(\er{n})$ vanishes on a set of positive measure and if $-\Delta \p + V \p = 0$ weakly, then $\p$ has a zero of infinite order.
\end{proposition}

The proof is written in \cite{FigGos92} under the assumption that $V \in L^{n/2}_{\rm{loc}}(\er{n})$ but after inspection, one realises that it only relies on \eqref{fig}. We remark that our assumption \eqref{main} is stronger than \eqref{fig}. This is because the square root is operator monotone, and therefore $$ \abs{V} \indic_{B_R} \leq \sqrt{ \ep_n (-\Delta)^{\f{3}{2} - 2\delta} +c_R } \leq \ep_n^{\f{2}{3}} (-\Delta) + c_R'.$$ For this reason, we will assume for the rest of the proof that $\p$ vanishes to infinite order at one point, which can be taken to be the origin without loss of generality.

\subsection*{Step 2. $\na \p$ and $\Delta \p$ vanish to infinite order as well.}

First we remark that if $\p \in L^2(\er{n})$, then vanishing to infinite order at the origin is equivalent to $\int_{B_1}\abs{x}^{-\tau} \abs{\p}^2 \d x$ being finite for every $\tau \geq 0$. Indeed, if $\p$ vanishes to infinite order at the origin, that is $\int_{B_{\ep}} \abs{\p}^2 \leq c_k \ep^k, $ then we get, after integrating over $\ep$, 
	\begin{align*}
		c_k \geq \int_0^1 \frac{\int_{B_{\ep}} \abs{\p}^2}{\ep^k} \d \ep & = \int_{B_1} \int_0^1 \abs{\p(x)}^2 \f{\indic_{\abs{x}\leq \ep}}{\ep^k}  \d \ep \, \d x \\
		& = \inv{k-1} \int_{B_1} \abs{\p(x)}^2 \pa{ \inv{\abs{x}^{k-1}} -1} \d x.
	\end{align*}
Conversely, if $\int_{\abs{x} \leq 1}\abs{x}^{-\tau} \abs{\p}^2$ is finite for every $\tau \geq 0$, then $$\ep^{-k} \int_{B_{\ep}} \abs{\p}^2 \leq \int_{B_{\ep}}\abs{x}^{-k} \abs{\p(x)}^2 \d x \leq \int_{B_1} \abs{x}^{-k}\abs{\p(x)}^2 \d x.$$ The finiteness of these integrals will play an important role later.

\begin{lemma}[Finiteness of weighted norms]\label{finite} \tx{ }

	i) Let $V \in L^1_{\rm{loc}} (\er{n})$ such that there exist $a<1$ and $c >0$ such that 
	\begin{equation*}
		\abs{V} \indic_{B_1} \leq a (-\Delta) +c.
	\end{equation*}
	Let $\p \in H^1_{\rm{loc}}(\er{n})$ satisfying $-\Delta \p + V \p=0$ weakly. If $\p$ vanishes to infinite order at the origin, then $\na \p$ as well.

	ii) 
	Let $V \in L^2_{\rm{loc}} (\er{n})$ such that there exist $a<1$ and $c >0$ such that 
	\begin{equation*}
		\abs{V}^2 \indic_{B_1} \leq a (-\Delta)^2 +c.
	\end{equation*}
	Let $\p \in H^2_{\rm{loc}}(\er{n})$ satisfying $-\Delta \p + V \p=0$. If $\p$ vanishes to infinite order at the origin, then $\na \p$ and $\Delta \p$ as well.
\end{lemma}

\begin{proof}
	$i$) We take $\ep \in (0,1/2]$ and define a smooth localisation function $\eta$ with support in $B_{2\ep}$, equal to $1$ in $B_{\ep}$, and such that $\abs{\na \eta} \leq c/\ep$ and $\abs{\Delta \eta} \leq c/\ep^2$.
	Multiplying the equation by $\eta^2 \ov{\p}$ and taking the real parts yields
 \begin{align*}
	 - \re \int V \abs{\eta \p}^2 & = - \re \int \ov{\p} \eta^2 \Delta \p = \re \int \na \p \cdot \na \pa{\eta^2 \ov{\p}} \\
	 = \int \abs{\eta \na \p}^2 & + \re \int \ov{\p} \na \p \cdot \na \eta^2 = \int \abs{\eta \na \p}^2 + \ud \int \na \abs{\p}^2 \cdot \na \eta^2 \\
	 = \int \abs{\eta \na \p}^2 & - \ud \int \abs{\p}^2 \Delta \eta^2.
 \end{align*}
	So by the assumption on $V$,
 \begin{align*}
	 \int \abs{\eta \na \p}^2 & \leq a \int \abs{\na \pa{\eta \p}}^2 + c \int \abs{\eta \p}^2+ \ud \int \abs{\p}^2 \Delta \eta^2 \\
	 & = a \int \abs{\eta \na \p}^2 + a \int \abs {\p \na \eta}^2 + \f{1-a}{2} \int \abs{\p}^2 \Delta \eta^2  + c \int \abs{\eta \p}^2.
 \end{align*}
	So, since $a<1$, we get
 \begin{align*}
	 \int_{B_{\ep}} \abs{\na \p}^2 \leq \int \abs{\eta \na \p}^2 \leq c_a \ep^{-2} \int_{B_{2\ep}} \abs{\p}^2 < c_a c_k 2^k \ep^{k-2},
 \end{align*}
	for any $k \geq 0$, where we used that $\p$ vanishes to infinite order. This proves the result.

	$ii$) By $i$), we know that $\na \p$ vanishes to infinite order at the origin. We take the same function $\eta$ as in $i$) and use the Schr\"odinger equation pointwise to get
 \begin{align*}
	 \int \abs{\eta \Delta \p}^2 & = \int \abs{V \eta \p}^2 \leq a \int \abs{\Delta \pa{\eta \p}}^2 + c \int \abs{\eta \p}^2 \\
	 & \leq a(1+\alpha) \int \abs{\eta \Delta \p}^2 + 2\pa{1+ \inv{\alpha}} \int \abs{\p \Delta \eta}^2 \\
	 & \hspace{0.35cm} + 4\pa{1+ \inv{\alpha}} \int \abs{\na \p \cdot \na \eta}^2 + c \int \abs{\eta \p}^2,
 \end{align*}
	for any $\alpha >0$. We take $\alpha$ such that $a(1+\alpha) < 1$ and thus 
 \begin{align*}
	 \int_{B_{\ep}} \abs{\Delta \p}^2 \leq \int \abs{\eta \Delta \p}^2 \leq c_a \ep^{-4} \int_{B_{2\ep}} \pa{\abs{\p}^2+\abs{\na \p}^2} < c_a c_k 2^k \ep^{k-4},
 \end{align*}
which proves the result.
\end{proof}

\subsection*{Step 3. Carleman estimate}
One common tool for unique continuation results is the Carleman estimate. We use here a new Carleman inequality with singular weights which we have recently proved in \cite[Corollary 1.2]{Garrigue19}. We define the weight function
 \begin{align}\label{est}
	  \phi(x) \df -\ln \ab{x} + (-\ln \ab{x})^{-1/2}.
 \end{align}
In dimension $n$, for any $\xi >0$ there exist constants $\kappa_{\xi,n}$ and $\tau_0 \ge 1$ such that for any $s \in \seg{0,1}$, any $\tau \ge \tau_0$ and any $u \in C^{\ii}_{\tx{c}}(B_{1} \backslash \acs{0},\C)$, we have \footnote{In the published version of the present article \cite{Garrigue18}, we used another Carleman estimate with the different weight $\phi(x) = -\ln \pa{ \ab{x} + \lambda \ab{x}^2}$, taken from \cite[Theorem 12]{Tataru04}. After publication, we realized that the estimate on $\Delta\pa{e^{\tau \phi} u}$ was not known for this weight. The estimate with the weight \eqref{est}, proved in \cite{Garrigue19}, is of remedies.  }
 \begin{align}\label{carl}
	 \tau^{\f{3}{2}-2s} \nor{ (-\Delta)^{(1-\xi)s} \pa{e^{\tau \phi} u}}{L^2(\R^n)} \le \kappa_{\xi,n}  \nor{e^{\tau \phi} \Delta u}{L^2(B_1)}. 
 \end{align}

\subsection*{Step 4. Proof that $\p =0$.}

We consider some number $\tau\geq 0$ (large), and we call $c$ any constant that does not depend on $\tau$. We take a smooth localisation function $\eta$, equal to $1$ in $B_{1/2} \subset \er{n}$, supported in $B_{1}$, and such that $0 \leq \eta \leq 1$. We take the weight function $\phi$ as in \eqref{carl}. It verifies $e^{\phi(x)}+ \abs{\na \phi} \leq c\abs{x}^{-1}$ and $\abs{\Delta \phi} \leq c \abs{x}^{-2}$ for $c$ sufficiently large.

In step 1, we have shown that $\p$ vanishes to infinite order at the origin and in step 2 we have deduced the same property for $\na \p$ and $\Delta \p$. Moreover, $$ \int_{B_1} \frac{\abs{\p(x)}^2}{\abs{x}^{\tau}} \d x + \int_{B_1} \frac{\abs{\na \p(x)}^2}{\abs{x}^{\tau}}\d x + \int_{B_1} \frac{\abs{\Delta \p(x)}^2}{\abs{x}^{\tau}} \d x < +\ii,$$
for all $\tau \geq 0$. All the integrals with $e^{\tau\phi}$ are finite as well and the following calculations are valid. In addition, from the Carleman inequality \eqref{carl}, we know that $e^{\tau\phi}\Psi$ belongs to $H^{3/2}\loc(\er{n})$ for all $\tau$.

By the assumption \eqref{main} on $V$, we have
\begin{equation*}
\nor{e^{\tau\phi} V \eta \p}{L^2(B_{1})} \leq \sqrt{\ep_{\delta,n}} \nor{(-\Delta )^{\frac{3}{4} - \delta} \pa{e^{\tau\phi} \eta\p}}{L^2(\er{n})} + c \nor{e^{\tau\phi} \eta\p}{L^2(B_{1})}.
\end{equation*}
Applying the Carleman estimate \eqref{carl} with $s=3/4$ and $\xi = 4\delta/3$, we get
\begin{equation*}
	\nor{(-\Delta )^{\frac{3}{4} - \delta} \pa{e^{\tau\phi}\eta \p}}{L^2(\er{n})} \leq \kappa_{4\delta/3,n} \nor{e^{\tau\phi}\Delta (\eta\p)}{L^2(B_{1})},
\end{equation*}
and hence
\begin{equation*}
	\nor{e^{\tau\phi} V \eta \p}{L^2(B_{1})} \leq \kappa_{4\delta /3,n} \sqrt{\ep_{\delta,n}}\nor{e^{\tau\phi}\Delta (\eta\p)}{L^2(B_{1})}+ c \nor{e^{\tau\phi} \eta\p}{L^2(B_{1})}.
\end{equation*}
Now we estimate
\begin{align*}
	& \nor{e^{\tau\phi}\Delta (\eta\p)}{L^2(B_{1})} \\
	& \qquad \leq  \nor{e^{\tau\phi} \eta \Delta \p}{L^2(B_{1})} + 2 \nor{e^{\tau\phi} \na \eta \cdot \na \p}{L^2(B_{1})} + \nor{e^{\tau\phi} \p \Delta \eta}{L^2(B_{1})} \\
	& \qquad \leq  \nor{e^{\tau\phi} V \eta \p}{L^2(B_{1})} + c\nor{e^{\tau\phi} \na \p}{L^2(B_{1}\backslash B_{1/2})} + c \nor{e^{\tau\phi} \p }{L^2(B_{1} \backslash B_{1/2})} \\
	& \qquad \leq  \kappa_{4\delta/3,n} \sqrt{\ep_{\delta,n}}  \nor{e^{\tau\phi} \Delta\pa{ \eta \p}}{L^2(B_{1})} + c  \nor{e^{\tau\phi} \eta \p}{L^2(B_{1})} + c e^{\tau \phi\pa{\ud}}.
\end{align*}
We take $$\ep_{\delta,n} = \f{1}{4\kappa_{4\delta/3,n}^2},$$ and move the term $\nor{e^{\tau\phi} \Delta (\eta\p)}{L^2(B_{1})}$ to the left side of the inequality, which yields
\begin{equation*}
	\nor{e^{\tau\phi}\Delta (\eta\p)}{L^2(B_{1})} \leq c \nor{e^{\tau\phi} \eta\p}{L^2(B_{1})}+ c e^{\tau\phi\pa{\ud}}.
\end{equation*}
But by the Carleman inequality \eqref{carl} applied with $s=0$, we have 
\begin{equation*}
   \nor{e^{\tau\phi} \eta\p}{L^2(B_{1})} \leq c \tau^{-\frac{3}{2}} \nor{e^{\tau\phi}\Delta (\eta\p)}{L^2(B_{1})},
\end{equation*}
so for $\tau$ big enough so that $c \tau^{-\frac{3}{2}} < 1/2$, we find
\begin{equation*}
	\nor{\eta\p}{L^2(B_{1/2})} \leq \nor{e^{\tau\pa{\phi(\cdot)-\phi\pa{\ud}}} \eta\p}{L^2(B_{1/2})} \leq c\tau^{-\frac{3}{2}}.
\end{equation*}
Eventually, letting $\tau \ra +\ii$, we get $\p = 0$ almost everywhere in $B_{1/2}$. We can then propagate this small region $B_{1/2}$, where $\p$ vanishes, to the whole space, as explained for instance in \cite{ReeSim4}.

\section{Proof of Corollary \ref{th2}}

We take $n=dN$. Let $R>0$ and $\p \in H^{3/2}(\er{dN})$. We apply the inequality \eqref{borne} to the function $x_i \mapsto \p(\dots, x_i, \dots)$ and then integrate over $x_1, \dots,x_{i-1},x_{i+1},\dots,x_N$ to get
\begin{equation*}
	\int_{B_R} \abs{v(x_i)}^2 \abs{\p}^2 \leq \ep_{\delta,d,N} \int_{\er{dN}} \abs{(-\Delta_{x_i})^{\frac{3}{4}-\delta} \p}^2 + c_{R} \int_{\er{dN}} \abs{\p}^2.
\end{equation*}
For $j \neq i$, applying \eqref{borne} with a radius $2R$, we have similarly 
\begin{equation*}
	\int_{B_R} \abs{w(x_i-x_j)}^2 \abs{\p}^2 \leq \ep_{\delta,d,N} \int_{\er{dN}} \abs{(-\Delta_{x_i})^{\frac{3}{4}-\delta} \p}^2 +c_R \int_{\er{dN}} \abs{\p}^2.
\end{equation*}
We consider the many-body potential 
\begin{equation}\label{pot}
V(x_1, \dots ,x_N) \df \sum_{i=1}^{N} v(x_i) + \sum_{1 \leq i \sle j \leq N} w(x_i-x_j),
\end{equation}
for which
 \begin{align*}
	 \abs{V}^2\indic_{B_R} & = \indic_{B_R}\abs{\sum_{i=1}^N v(x_i) + \sum_{1 \leq i \sle j \leq N} w(x_i-x_j)}^2 \\
	 & \leq \f{N(N+1)}{2} \pa{\sum_{i=1}^N \indic_{B_R}\abs{v(x_i)}^2 + \sum_{1 \leq i \sle j \leq N} \indic_{B_R}\abs{w(x_i-x_j)}^2} \\
	 & \leq \f{N(N+1)^2}{4} \pa{ \ep_{\delta,d,N} \sum_{i=1}^N (-\Delta_{x_i})^{\f{3}{2}-2\delta} + N c_{R}} \\
	 & \leq \f{N(N+1)^2}{4} \pa{ \ep_{\delta,d,N} (-\Delta)^{\f{3}{2}-2\delta} + N c_{R}},
 \end{align*}
where in the last inequality we have used that $$\sum_{i=1}^N |k_i|^{3-4\delta}\leq \pa{\sum_{i=1}^N|k_i|^2}^{\f{3}{2}-2\delta}.$$ Thus we can take $$ \ep_{\delta,d,N} = \frac{4\ep_{\delta,dN}}{N(N+1)^2}= \f{1}{N(N+1)^2 \kappa_{4\delta/3,dN}^2},$$ and we obtain the result by applying Theorem \ref{mainthm}.

To finish, we prove that the assumption that ${v,w \in L^p_{\rm{loc}}(\er{d})}$ with $p > \max(2d/3,2)$ implies \eqref{borne}. This is very classical \cite{Kato51,LieLos01}. First let $s \in (0,d/2)$, let $v \in L^{\frac{d}{2s}}_{\rm{loc}}(\er{d})$, $R > 0$ and $u \in H^{s}(\er{d})$ supported in $B_R \subset \er{d}$. We have ${v = v \indic_{\abs{v}> M} + v \indic_{\abs{v}< M}}$, so
\begin{align*}
\int_{\er{d}} \abs{v} \abs{u}^2 & \leq \int_{\er{d}} \abs{v \indic_{\acs{\abs{v}> M} \cap B_R}} \abs{u}^2 + M \int_{\er{d}} \abs{u}^2 \\
& \leq \nor{v \indic_{\abs{v}> M}}{L^{\frac{d}{2s}}(B_R)} \nor{u}{L^{\frac{2d}{d-2s}}}^2 + M \no{u}^2 \\
& \leq c_{d,s}\nor{v \indic_{\abs{v}> M}}{L^{\frac{d}{2s}}(B_R)}  \no{(-\Delta)^{\frac{s}{2}}u}^2 + M \no{u}^2,
\end{align*}
where, in the last line, we have used the Sobolev inequality. By dominated convergence, $\nor{v \indic_{\abs{v}> M}}{L^{\frac{d}{2s}}(B_R)}$ tends to $0$ when $M \ra +\ii$. We can do a similar treatment for $w$. Therefore, this proves that for $s \in (0,d/2)$ and $q \geq 1$, if $v,w \in L^{\frac{qd}{2s}}_{\rm{loc}}(\er{d})$, then for any $R>0$ and any $\ep > 0$, there is $c_{\ep,R}$ such that
 \begin{equation*}
 \abs{v}^q \indic_{B_R} + \abs{w}^q \indic_{B_R} \leq \ep (-\Delta)^s +c_{\ep,R} \qquad \tx{ in $\er{d}$ }.
 \end{equation*}

	For the case $d \in \acs{1,2}$, we need $v,w \in L^2_{\rm{loc}}(\er{d})$ because we use $\abs{V}^2$. We have the Sobolev embedding ${H^{3/2}(\er{d}) \hookrightarrow L^{\ii}(\er{d})}$, and the argument is the same.


\bibliographystyle{siam}

\end{document}